\documentclass[11pt]{amsart}

\usepackage{amsmath}

\usepackage{amssymb}
\usepackage{amsthm}
\usepackage{latexsym}
\usepackage{pstricks}
\date{}

\def\CC{{\NZQ C}}

\def\opn#1#2{\def#1{\operatorname{#2}}} 
\opn\chara{char} \opn\length{\ell} \opn\pd{pd} \opn\rk{rk}
\opn\projdim{proj\,dim} \opn\injdim{inj\,dim} \opn\rank{rank}
\opn\depth{depth} \opn\codepth{codepth} \opn\grade{grade}
\opn\height{height} \opn\embdim{emb\,dim} \opn\codim{codim}

\opn\Tr{Tr} \opn\bigrank{big\,rank}
\opn\superheight{superheight}\opn\lcm{lcm}
\opn\trdeg{tr\,deg}%
\opn\reg{reg} \opn\lreg{lreg} \opn\skel{skel} \opn\Gr{Gr}
\opn\dim{dim} \opn\arithdeg{arithdeg}
\opn\dev{dev}

\opn\diam{diam}

\opn\div{div} \opn\Div{Div} \opn\cl{cl} \opn\Cl{Cl}
\opn\Spec{Spec} \opn\Supp{Supp} \opn\supp{supp} \opn\Sing{Sing}
\opn\Ass{Ass}
\opn\Ann{Ann} \opn\Rad{Rad} \opn\Soc{Soc}
\opn\Sym{Sym} \opn\Ker{Ker} \opn\Coker{Coker} \opn\Im{Im}
\opn\Hom{Hom} \opn\Tor{Tor} \opn\Ext{Ext} \opn\End{End}
\opn\Aut{Aut} \opn\id{id} \opn\ini{in} \opn\tr{tr}

\opn\nat{nat}\opn\it{it}
\opn\pff{proof}
\opn\Pf{proof} \opn\GL{GL} \opn\SL{SL} \opn\mod{mod} \opn\ord{ord}
\opn\aff{aff} \opn\con{conv} \opn\relint{relint} \opn\st{st}
\opn\lk{lk} \opn\cn{cn} \opn\core{core} \opn\vol{vol}
\opn\link{link} \opn\star{star} \opn\skel{skel}
\opn\cone{cone} \opn\star{star} \opn\skel{skel}

\opn\gr{gr}

\def\pot#1#2{#1[\kern-0.28ex[#2]\kern-0.28ex]}

\opn\dirlim{\underrightarrow{\lim}}
\opn\inivlim{\underleftarrow{\lim}}
\let\ol=\overline
\let\wt=\widetilde

\def\Implies{\ifmmode\Longrightarrow \else
     \unskip${}\Longrightarrow{}$\ignorespaces\fi}
\def\implies{\ifmmode\Rightarrow \else
     \unskip${}\Rightarrow{}$\ignorespaces\fi}
\def\iff{\ifmmode\Longleftrightarrow \else
     \unskip${}\Longleftrightarrow{}$\ignorespaces\fi}

\let\:=\colon
\newtheorem{Theorem}{Theorem}[section]
\newtheorem{Lemma}[Theorem]{Lemma}

\newtheorem{Proposition}[Theorem]{Proposition}
\newtheorem{Remark}[Theorem]{Remark}

\newtheorem{Example}[Theorem]{Example}

\newtheorem{Algorithm}[Theorem]{Algorithm}
\let\epsilon\varepsilon
\let\phi=\varphi
\let\kappa=\varkappa
\def\qed{\ifhmode\textqed\fi
   \ifmmode\ifinner\quad\qedsymbol\else\dispqed\fi\fi}
\def\textqed{\unskip\nobreak\penalty50
    \hskip2em\hbox{}\nobreak\hfil\qedsymbol
    \parfillskip=0pt \finalhyphendemerits=0}
\def\dispqed{\rlap{\qquad\qedsymbol}}

\opn\Gin{Gin}

\def\GG{{\mathcal G}}

\def\CC{{\mathcal C}}

\def\FF{{\mathcal F}}
\def\MM{{\mathcal M}}

\opn\inii{in} \opn\inim{inm} \opn\rate{rate}

\numberwithin{equation}{section}

\begin{document}
\title{Cohen-Macaulay binomial edge ideals of small deviation}

\author[Giancarlo Rinaldo]{Giancarlo Rinaldo}
\address[Giancarlo Rinaldo]{Dipartimento di Matematica, 
Universita' di Messina, 
Viale Ferdinando Stagno d'Alcontres, 31
98166 Messina, Italy.}
\email{giancarlo.rinaldo@tiscali.it}

\date{\today}

\begin{abstract}
We classify all binomial edge ideals that are complete intersection and  Cohen-Macaulay almost complete intersection. We also describe an algorithm and provide an implementation to compute primary decomposition of binomial edge ideals.
\end{abstract}

\maketitle

\section*{Introduction}


In 2010, binomial edge ideals were introduced in \cite{HH} and appeared independently also in \cite{MO}. Let $S = K[x_1,\ldots, x_n, y_1,\ldots, y_n]$ be the polynomial ring in $2n$ variables with coefficients in a field $K$. Let $G$ be a graph on vertex set $[n]$. For each edge $\{i,j\}$ of $G$ with $i < j$, we associate a binomial $f_{ij} = x_iy_j - x_jy_i$. The ideal $J_G$ of $S$ generated by $f_{ij} = x_iy_j - x_jy_i$ such that $i<j$ , is called \textit{the binomial edge ideal} of $G$. Any ideal generated  by a set of $2$-minors  of a $2\times n$-matrix of indeterminates may be viewed as the binomial edge ideal of a graph. 


Algebraic properties of binomial edge ideals in terms of properties of the underlying graph were studied in \cite{HH}, \cite{CR}, \cite{HEH} and \cite{RR}. In \cite{HEH} and \cite{RR} the authors considered the Cohen-Macaulay property of these graphs.  However, the classification of Cohen-Macaulay binomial edge ideals in terms of the underlying graphs is still widely open and, as in the case of monomial edge ideals introduced in \cite{Vi2}, it seems rather hopeless to give a full classification.

In this paper we consider Cohen-Macaulay and unmixed binomial edge ideals $J_G$ with small deviation, namely the difference between the minimum number of the generators and the height of $J_G$ is less than or equal to $2$. 

Section \ref{sec:pre} contains some preliminaries and notions that we use in the paper.
In the beginning of Section \ref{sec:main} we give a complete classification of the complete intersection binomial edge ideals (Theorem \ref{the:dev0}), that is the case of deviation $0$. This results is a consequence of Corollary 1.2 of \cite{HEH}. We also observe that in general the almost complete intersection, namely deviation $1$, binomial edge ideals are not unmixed. A nice example is the claw graph (see Example  \ref{exa:claw}). In Theorem \ref{the:dev1} we give a complete classification of Cohen-Macaulay binomial edge ideals that are almost complete intersection and we show that this set coincides with the set of  unmixed binomial edge ideals that are almost complete intersection. 

In Section \ref{sec:algo} we describe an algorithm to compute the primary decomposition of $J_G$ and provide an implementation  in CoCoA (see \cite{Co}) that is freely downloadable (see \cite{R}). Thanks to this we computed Examples \ref{exa:dev2} and \ref{exa:dev2bis} that are unmixed binomial edge  ideals of deviation $2$ that are not Cohen-Macaulay.

\section{Preliminaries}\label{sec:pre}
In this section we recall some concepts and notations on graphs and on simplicial
complexes that we will use in the article.

Let $G$ be a simple graph with vertex set $V(G)$ and the edge set $E(G)$. A subset $C$ of $V(G)$ is called a \textit{clique} of $G$ if for all $i$ and $j$ belonging to $C$ with $i \neq j$ one has $\{i, j\} \in E(G)$. A vertex of a graph is called cutpoint if the removal of the vertex increases the number of connected components. A vertex $v$ is a cut point of a graph $G$ if and only if there exist $u,w\in V(G)$ such that $v$ is in every path connecting $u$ and $w$ (see Theorem 3.1 of \cite{Ha}).
A subgraph $H$ of $G$ \textit{spans} $G$ if $V(H)=V(G)$. In a connected graph $G$ a \textit{chord} of a tree $T$ that spans $G$ is an edge of $G$ not in $T$. The number of chords of any spanning tree of a connected graph $G$, namely $m(G)$, is called the cycle rank of $G$ and is $m(G)=|E(G)|-|V(G)|+1$ (see Corollary 4.5(a) of \cite{Ha}). If $G$ has $c$ components than $m(G)=|E(G)|-|V(G)|+c$ (see Corollary 4.5(b) of \cite{Ha}).

Let $v\not\in V(G)$. The \textit{cone} of $v$ on $G$, namely $\cone(v,G)$, is the graph with vertices $V(G)\cup \{v\}$ and edges $E(G)\cup \{\{u,v\}:u\in V(G)\}$.

Let $G_1$ and $G_2$ be graphs. We set $G=G_1\cup G_2$  (resp. $G=G_1\sqcup G_2$ where $\sqcup$ is disjoint union) where $G$ is the graph with $V(G)=V(G_1)\cup V(G_2)$  (resp. $V(G)=V(G_1)\sqcup V(G_2)$) and $E(G)=E(G_1)\cup E(G_2)$ (resp. $E(G)=E(G_1)\sqcup E(G_2)$).

Set $V = \{x_1, \ldots, x_n\}$. A \textit{simplicial complex}
$\Delta$ on the vertex set $V$ is a collection of subsets of $V$
such that
\begin{enumerate}
\item[(i)] $\{x_i\} \in \Delta$  for all $x_i \in V$;
\item[(ii)] $F \in \Delta$ and $G\subseteq F$ imply $G \in \Delta$.
\end{enumerate}
An element $F \in \Delta$ is called a \textit{face} of $\Delta$. A maximal face of $\Delta$  with respect to inclusion is called a \textit{facet} of $\Delta$.
A vertex $i$ of $\Delta$ is called a free vertex of $\Delta$ if $i$ belongs to exactly one facet.

If $\Delta$ is a simplicial complex  with facets $F_1, \ldots, F_q$, we call $\{F_1, \ldots, F_q\}$ the facet set of $\Delta$ and we denote it by $\FF(\Delta)$.

The \textit{clique complex} $\Delta(G)$ of $G$ is the simplicial complex whose faces are the cliques of $G$. Hence a vertex $v$ of a graph $G$ is called \textit{free vertex} if it belongs to only one clique of $\Delta(G)$.

We need notations and results  from \cite{HH} (section 3) that we recall for the sake of completeness.

Let $T\subseteq [n]$, and let $\ol{T}=[n]\setminus T$. Let $G_1,\ldots,G_{c(T)}$ be the connected components of the induced subgraph on $\ol{T}$, namely $G_{\ol{T}}$. For each $G_i$, denote by $\wt{G}_i$ the complete graph on the vertex set $V(G_i)$. We set 
\begin{equation}\label{eq:prime}
P_T(G)=(\bigcup_{i\in T}\{x_i,y_i\},J_{\wt{G}_1},\ldots,J_{\wt{G}_{c(T)}} ), 
\end{equation}
$P_T(G)$ is a prime ideal. Then $J_G=\bigcap_{T\subset [n]}P_T(G)$. If there is no confusion possible, we write simply $P_T$ instead of $P_T(G).$ Moreover, $\height P_T=n+|T|-c(T)$ (see \cite[Lemma 3.1]{HH}).
 We denote by $\MM(G)$ the set of minimal prime ideals of $J_G$.

If each $i\in T$ is a cut point of the graph $G_{\ol{T}\cup \{i\}}$, then we say that $T$ has {\em cutpoint property for $G$.} We denote by $\mathcal{C}(G)$ the set of all $T\subset V(G)$ such that $T$ has cutpoint property for $G$.

\begin{Lemma}\label{lem:cutpoint}\cite{HH}
 $P_T(G)\in \MM(G)$ if and only if $T\in\mathcal{C}(G)$.
\end{Lemma}

\begin{Lemma}\label{lem:unmixwd}\cite{RR} Let $G$ be a connected graph. Then $J_G$ is unmixed if and only if for all $T\in\mathcal{C}(G)$ we have $c(T)=|T|+1$.
\end{Lemma}

\section{Complete intersection and almost complete intersection}\label{sec:main}
Let $S$ be a standard graded polynomial ring over a field $K$. For a homogeneous ideal $J \subseteq S$, $J$ is called a complete intersection ideal (resp. an almost complete intersection ideal)  if $J$ is minimally generated by $\height J$ (resp. $\height J + 1$) elements. A homogeneous ideal has deviation $\dev(J)$ if it is minimally generated by $\height I +\dev(J)$ elements. 
Throughout this section let $S=K[\{x_i,y_i\}:i\in V(G)]$, $J_G$ be the binomial edge ideal of a graph $G$ and $\mu(J_G)$ the minimal number of generators of $J_G$. 

A nice combinatorial interpretation of $\dev(J_G)$ is given by the following
\begin{Remark}\label{rem:mg}
Suppose that $\height J_G=\height P_\emptyset$. Then $\height J_G=n-c$ where $c$ are the connected components of $G$. Therefore \[\dev(J_G)=\mu(J_G)-n+c=m(G).\]
\end{Remark}

\begin{Theorem}\label{the:dev0}
Let $G$ be a graph. Then $J_G$ is complete intersection if and only if each component of $G$ is a path graph.
\end{Theorem}
\begin{proof}
 Let $G=\bigcup_{i=1}^c G_i$ where $G_i$ are the connected components of $G$. Let $n_i=|V(G_i)|$ for $i=1,\ldots,c$. Since $G_i$ is connected it has at least $n_i-1$ edges, namely the number of edges of a tree. Hence
 \[
  \mu(J_G)\geq \sum_{i=1}^c (n_i-1)=n-c.
 \]
Since $J_G$ is a complete intersection then it is unmixed, hence $\height J_G=\height P_\emptyset$. By Remark \ref{rem:mg} and since complete intersection implies $\dev(J_G)=0$ we obtain that $\mu(J_G)=n-c$. 
Therefore every connected component $G_i$ is a tree. Now the proof is a consequence of Corollary 1.2 of \cite{HEH}.
\end{proof}

In general almost complete intersection binomial edge ideals are not unmixed as the following example shows
\begin{Example}\label{exa:claw}
 Let $G$ be the graph on $4$ vertices and edges 
 \[
 \{\{1,2\},\{1,3\},\{1,4\}\},
 \]
 namely the claw graph. We observe that
 \[
  J_G=P_\emptyset\cap P_{\{1\}}
 \]
where $\height P_\emptyset=3$ and $\height P_{\{1\}}=2$. 
\end{Example}

\begin{Remark}\label{rem:conn}
 If $J_G$ is an unmixed almost complete intersection binomial edge ideal with $c$ components we have that $G$ has $c-1$ components that are path graphs and $1$ that contains only one cycle, namely a unicyclic graph. The proof is similar to the proof of Theorem \ref{the:dev0}.
\end{Remark}
Thanks to Remark \ref{rem:conn} we assume from now on that $G$ is a connected unicyclic graph.
\begin{Proposition}\label{pro:tri}
 Let $\GG_3$ be the set of graphs such that for all $G\in \GG_3$ we have
\[
 V(G)=\{u_1,\ldots,u_r,v_1,\ldots,v_s,w_1,\ldots,w_t\}
\]
with $r\geq 1$, $s\geq 1$, $t\geq 1$ and edge set 
\begin{equation*}
\begin{split}
E(G) =&\{\{u_i,u_{i+1}\}:i=1,\ldots,r-1\}\cup\{v_i,v_{i+1}\}:i=1,\ldots,s-1\}\cup\\
 & \cup\{w_i,w_{i+1}\}:i=1,\ldots,t-1\} \cup\{\{u_1,v_1\},\{u_1,w_1\},\{v_1,w_1\}\}.  
\end{split}
\end{equation*}
Then $S/J_G$ is Cohen-Macaulay for all $G\in \GG_3$.
\end{Proposition}
\begin{proof}
 If $r=s=t=1$ then $G$ is a complete graph. Hence it is Cohen-Macaulay. Since $u_1$, $v_1$ and $w_1$ are free vertices in $\Delta(G)$, by Theorem 2.7 of \cite{RR} the assertion follows easily.
\end{proof}

\begin{Lemma}\label{lem:fan}
 Let $G$ be the graph with vertex set $V(G)=\{1,\ldots,6\}$ and edge set 
 $E(G)=\{\{1,2\},\{2,3\},\{2,4\},\{2,5\},\{3,4\},\{4,5\},\{5,6\}\}$.
 Then $S/J_G$ is Cohen-Macaulay.
\end{Lemma}
\begin{proof}
Let $H_1$ be the graph with one vertex, $V(H_1)=\{1\}$, $H_2$ the path graph with edges $\{3,4\}$ and $\{4,5\}$   
and let $G_1=\cone(2, H_1\sqcup H_2)$. By Theorem 3.8 of \cite{RR} $G_1$ is Cohen-Macaulay. Now let $G_2$ be the complete graph on the vertices $V(G_2)=\{5,6\}$. Then $5$ is a free vertex in  $\Delta(G_1)$ and $\Delta(G_2)$. By Theorem 2.7 of \cite{RR} the assertion follows. 
\end{proof}

\begin{Proposition}\label{pro:qua}
 Let $\GG_4$ be the set of graphs such that for all $G\in \GG_4$ we have
 \[
V(G)=\{u_1,u_2,u_3,\ldots,u_r,v_1,v_2,v_3,\ldots,v_s\}
 \]
with $r\geq 3$ and $s\geq 3$ and edge set 
\begin{equation*}
\begin{split}
E(G)=&\{\{u_i,u_{i+1}\}:i=1,\ldots,r-1\}\cup\{\{v_i,v_{i+1}\}:i=1,\ldots,s-1\}\cup\\
&\cup \{\{u_1,v_1\},\{u_2,v_2\}\}.    
\end{split}
\end{equation*}
Then $S/J_G$ is Cohen-Macaulay for all $G\in \GG_4$.
\end{Proposition}
\begin{proof}
 Let $s=r=3$.  We observe that 
 \[
 \mathcal{C}(G)=\{\emptyset, \{u_2 \}, \{v_2\}, \{u_2,v_1\}, \{u_1,v_2\}, \{u_2,v_2\}\}
 \]
 and $J_G$ is unmixed with $\dim S/J_G=7$ with $S=K[\{x_i,y_i\}: i\in V(G)]$ . We need to show that $\depth S/J_G\geq 7$. Let 
\[
J_H=P_\emptyset \cap P_{\{u_2\}}\cap P_{\{v_2\}}\cap P_{\{u_2,v_1\}} \cap P_{\{u_2,v_2\}}
\]
and
\[
J_{H'}=P_\emptyset \cap P_{\{u_2\}}\cap P_{\{v_2\}}\cap P_{\{u_1,v_2\}} \cap P_{\{u_2,v_2\}}
\]
be  binomial edge ideals on $S$. The graphs $H$ and $H'$ are both isomorphic to the graph described in Lemma \ref{lem:fan}. Hence they are Cohen-Macaulay with dimension equal to $7$. Let
\[
 J_{H''}=J_H+J_{H'}\subset S.
\]
We observe that $H''=G_1\cup G_2\cup G_3$ where $G_1$ is the complete graph on the vertex set $\{u_1,u_2,v_1,v_2\}$ and $G_2$ (resp. $G_3$) is the complete graph on the vertex set $\{u_2,u_3\}$ (resp. $\{v_2,v_3\}$). By Theorem 2.7 of \cite{RR}, applied twice (or by Theorem 1.1 of \cite{HEH}) $S/J_{H''}$ is Cohen-Macaulay with dimension equal to $7$. Thanks to depth Lemma applied to the following exact sequence 
\[
0\longrightarrow S/J_G \longrightarrow S/J_H\oplus S/J_{H'}\longrightarrow S/J_{H''}\longrightarrow 0
\]
we obtain that depth $S/J_G$ is greater than or equal to $7$. Hence $S/J_G$ is Cohen-Macaulay, too. The assertion follows by Theorem 2.7 of \cite{RR} observing that $u_3$ and $v_3$ are free vertices in $\Delta(G)$.  
\end{proof}

\begin{Theorem}\label{the:dev1}
 Let $G$ be a graph and $J_G$ is an almost complete intersection binomial edge ideal. The following conditions are equivalent:
 \begin{enumerate}
  \item $G$ is in $\GG_3\cup \GG_4$  defined as in Propositions \ref{pro:tri} and \ref{pro:qua};
  \item $S/J_G$ is CM; 
  \item $S/J_G$ is unmixed.
 \end{enumerate}
\end{Theorem}
\begin{proof}

 $1) \Rightarrow 2)$. It follows by Propositions \ref{pro:tri} and \ref{pro:qua}.
 
 $2) \Rightarrow 3)$. Always true.
 
 $3) \Rightarrow 1)$. Suppose that $J_G$ is unmixed. By Lemma \ref{lem:unmixwd} if $G$ is unicyclic then
 \begin{equation}\label{eq:unicyclic}
  G=C_l\cup \left(\bigcup_{i=1}^r P_i\right) 
 \end{equation}
with $0\leq r\leq l$, where $C_l$ is a cycle of length $l$ and  for all $1\leq i \leq r$, $P_i$ is a path graph,
$|V(P_i)\cap V(C_l)|=1$ and $|V(P_i)\cap V(P_j)|=0$ for all $j\neq	i$. Suppose that $G$ is not in $\GG_3\cup \GG_4$. Then it is an element of one of the following sets:
 \begin{enumerate}
  \item [$\GG_3'$] $=\{G \mbox{ satisfies \eqref{eq:unicyclic} with }l=3\}\setminus \GG_3$;
  \item [$\GG_4'$] $=\{G \mbox{ satisfies \eqref{eq:unicyclic} with }l=4\}\setminus \GG_4$;
  \item [$\GG_>'$] $=\{G \mbox{ satisfies \eqref{eq:unicyclic} with }l\geq 5\}$.
 \end{enumerate}

$G$ does not belong to $\GG_3'$ since it is the empty set. Let $G\in \GG_4'$. We have that there are two vertices $v$ and $v'$ in $C_4$ that are not adjacent and have the same degree, that is either $2$ or $3$. We observe that $T=\{v,v'\}$ has the cutpoint property. If the degree is $2$ then $c(T)=|T|$, if the degree is $3$ then $c(T)=|T|+2$. In both cases $J_G$ is not unmixed by Lemma \ref{lem:unmixwd}. Contradiction. 

Let $G\in \GG'_>$ and let $V(C_l)=\{i_1,i_2,\ldots,i_l\}$ such that $\{i_j,i_{j+1}\}$ is an edge of $G$ with $j=1,\ldots,l-1$ and $\{i_1,i_l\}$ is an edge of $G$, too. Since $G$ is unicyclic then $\{i_1,i_3\}$ has the cutpoint property. In fact $G_{V\setminus\{i_1,i_3\}}$ has at least two connected components one containing the vertex $i_2$ and another containing the vertices $\{i_4,\ldots,i_l\}$. Since $J_G$ is unmixed these components are exactly $3$. We may assume without loss of generality that $i_1$ has degree $3$ and $i_3$ has degree $2$. 
 
 By the same argument also $\{i_2,i_4\}$ has the cutpoint property and either $i_2$ or $i_4$ has degree $3$. 
 Suppose $i_4$ has degree $3$. Then $\{i_1,i_4\}$ has the cutpoint property and $G_{V\setminus\{i_1,i_4\}}$ has $4$ connected components. Contradiction. Hence $i_2$ must have degree $3$.  Also $\{i_3,i_l\}$ has the cutpoint property and since $i_3$ has degree $2$ then $i_l$ has degree $3$. Since $\{i_2,i_l\}$  has the cutpoint property and $G_{V\setminus\{i_1,i_3\}}$ has $4$ connected components we obtain a contradiction. 
 \end{proof}

\section{An algorithm to compute primary decomposition}\label{sec:algo}
In this section we describe Algorithm \ref{alg:pd} that summarizes the results of Lemma \ref{lem:cutpoint} and Proposition 2.1 of \cite{RR} and provide an implementation  in CoCoA (see \cite{Co}) that is freely downloadable (see \cite{R}). This tool help the research of unmixed binomial edge ideals of deviation greater than or equal to $2$. 


\medskip

\begin{Algorithm}[Computation of $\CC(G)$]\label{alg:pd}\mbox{}
\begin{description}
 \item [Input] A simple connected graph $G$ with $V(G)=[n]$.
 \item [Output] The set $\CC(G)$.
\end{description}
\begin{enumerate}
 \item [1.] $S:=\{1,\ldots,n\}\setminus\{$free vertices of $\Delta(G)\}$ 
 \item [2.] $\CC(G)=\{\emptyset\}$\;
 \item [3.] For each {$T\subset S$ and $1\leq |T|\leq n-2$ with $T=\{v_1,\ldots,v_r\}$} do

\begin{enumerate}
  \item[3.1] If $c(T)>1$ then
  \begin{enumerate}
    \item[] $i:=1$
    \item[3.1.1] While $c(T\setminus\{v_i\})<c(T)$ and $i\leq r$ do 
    \begin{enumerate}
	 \item[] $i:=i+1$
    \end{enumerate}

     \item[3.1.2] If $i>r$ then 
     \begin{enumerate}
	 \item[] $\CC(G):=\CC(G)\cup\{T\}$
     \end{enumerate}

  \end{enumerate}
\end{enumerate}
 \item[4.] Return $\CC(G)$
\end{enumerate}
\end{Algorithm}

%
%
%

We give a description of the Algorithm \ref{alg:pd}. 
\begin{itemize}
 \item Line 1. By Proposition 2.1 of \cite{RR} we can avoid all the free vertices of $\Delta(G)$ in the computation of $\CC(G)$.
 \item Line 2. The empty set is always in $\CC(G)$ by Lemma \ref{lem:cutpoint}.
 \item Line 3. $T$ can be any subset of $S$ by Lemma \ref{lem:cutpoint}. Nevertheless Since $T$ has the cutpoint property, $c(T)\geq 2$ (see line 3.1). Therefore the maximum cardinality of $T$ is $n-2$ where the $2$ connected components are isolated vertices (if such $T$ exists).
 \item Line 3.1. We observe that if $c(T)=1$ then $c(T\setminus \{v_i\})=c(T)$ for all $v_i\in T$. Hence we discard such $T$.
 \item Line 3.1.1-3.1.2. We check if exists a $v_i\in T$ that does not satisfy the condition $c(T\setminus\{v_i\})<c(T)$. If such $v_i$ exists interrupt the current computation  otherwise add the new set to $\CC(G)$ (line 3.1.2). 
\end{itemize}

Thanks to Algorithm \ref{alg:pd} we found some unmixed binomial edge ideals of deviation  $2$ that are not Cohen-Macaulay. We provide two examples. The first one is interesting since it is a bipartite graph. The second one since it has induced $5$-cycle subgraphs.  

 \begin{figure}[hbt]
\begin{center}

\psset{unit=.8cm}

\begin{pspicture}(0,0)(14,5)

\rput(2,2){$\bullet$}
\rput(3,2){$\bullet$}
\rput(4,1){$\bullet$}
\rput(4,2){$\bullet$}
\rput(4,3){$\bullet$}
\rput(5,2){$\bullet$}
\rput(6,2){$\bullet$}

\rput(8,2){$\bullet$}
\rput(9,2){$\bullet$}
\rput(10,1){$\bullet$}
\rput(10,2){$\bullet$}
\rput(10,3){$\bullet$}
\rput(11,1){$\bullet$}
\rput(11,3){$\bullet$}

\rput(12,1){$\bullet$}
\rput(12,3){$\bullet$}

\put(1.5,1.8){$1$}
\put(2.8,1.5){$2$}
\put(3.9,0.5){$5$}
\put(3.6,1.8){$4$}
\put(3.9,3.2){$3$}
\put(5,1.5){$6$}
\put(6.3,1.8){$7$}

\put(7.5,1.8){$1$}
\put(8.8,1.5){$2$}
\put(9.9,0.5){$5$}
\put(10.2,1.8){$4$}
\put(9.9,3.2){$3$}
\put(10.9,3.2){$6$}
\put(10.9,0.5){$7$}
\put(11.9,3.2){$8$}
\put(11.9,0.5){$9$}

\psline(2,2)(3,2)
\psline(5,2)(6,2)

\psline(3,2)(4,3)
\psline(3,2)(4,1)

\psline(4,1)(5,2)

\psline(4,3)(5,2)

\psline(4,3)(4,1)

\psline(8,2)(9,2)
\psline(10,1)(12,1)
\psline(10,3)(12,3)

\psline(9,2)(10,3)
\psline(9,2)(10,1)

\psline(10,3)(10,1)
\psline(11,3)(11,1)


\end{pspicture}
 
\end{center}
\caption{}\label{dev2} 
\end{figure}

\begin{Example}\label{exa:dev2}
Let $J_G$ the binomial edge ideal associated to the graph with $7$ vertices in figure \ref{dev2}. Then
\[
\CC(G)=\{\emptyset,\{2\},\{6\},\{2,6\},\{3,5\},\{2,4,6\}\}
\]
and is unmixed with $\dim S/J_G=8$. Using CoCoA (see \cite{Co}) the $\depth=S/J_G=7$.
\end{Example}

\begin{Example}\label{exa:dev2bis}
Let $J_G$ the binomial edge ideal associated to the graph with $9$ vertices in figure \ref{dev2}. Then
\begin{equation*}
\begin{split}
\CC(G)=& \{\emptyset,\{2\},\{6\},\{7\},\{2,6\},\{2,7\},\{3,5\},\{3,7\},\{5,6\},\{6,7\}, \\
  & \{2,3,7\},\{2,4,6\},\{2,4,7\},\{2,5,6\},\{2,6,7\},\{3,5,6\},\{3,5,7\},\\
  & \{2,4,6,7\}\}
\end{split}
\end{equation*}
and is unmixed with $\dim S/J_G=10$. Using CoCoA (see \cite{Co}) the $\depth=S/J_G=9$.
\end{Example}

\bibliographystyle{plain}

\end{document}